\documentclass[11pt,reqno]{amsart}

\usepackage{indentfirst}
\usepackage{amssymb}
\usepackage{amsmath}
\usepackage{graphics}
\usepackage{epsfig}
\usepackage{MnSymbol,wasysym}

\usepackage[utf8]{inputenc}
\usepackage{latexsym}
\usepackage{graphicx}

\usepackage[pagewise]{lineno}
 
\usepackage[normalem]{ulem}

\usepackage{hyperref}

\usepackage[T1]{fontenc}
\usepackage{kantlipsum}
\usepackage[usenames,dvipsnames]{xcolor}
\usepackage[breakable, theorems, skins]{tcolorbox}
\tcbset{enhanced}

\catcode`@=11 \@addtoreset{equation}{section}
\renewcommand\theequation{\thesection.\@arabic\c@equation}
\catcode`@=12

\newcommand{\RR}{\mathbb{R}}

\newtheorem{theorem}{Theorem}[section]
\newtheorem{lemma}[theorem]{Lemma}
\newtheorem{proposition}[theorem]{Proposition}

\theoremstyle{remark}
\newtheorem{remark}{Remark}[section]
\theoremstyle{definition}

\newtheorem{definition}{Definition}[section]

\numberwithin{equation}{section}
\usepackage{hyperref}

\renewcommand{\a}{\alpha}

\newcommand{\Ll}{{\mathcal L}}

\begin{document}

\title[Dynamics for the energy-critical heat equation]{Global dynamics for the energy-critical nonlinear heat equation}

\author[M. Ikeda]{Masahiro Ikeda}
\address[M. Ikeda]{Graduate School of Information Science and Technology, Osaka University, 1-5, Yamadaoka, Suita-shi, Osaka 565-0871, Japan/Center for Advanced Intelligence Project RIKEN, 1-4-1, Nihonbashi, Chuo-ku, Tokyo 103-0027, Japan.}
\email{ikeda@ist.osaka-u.ac.jp/masahiro.ikeda@a.riken.jp}

\author[C. J. Niche]{C\'esar J. Niche}
\address[C.J. Niche]{Departamento de Matem\'atica Aplicada, Instituto de Matem\'atica. Universidade Federal do Rio de Janeiro, CEP 21941-909, Rio de Janeiro - RJ, Brazil}
\email{cniche@im.ufrj.br}

\author[G. Planas]{Gabriela Planas}
\address[G. Planas]{Departamento de Matem\'atica, Instituto de Matem\'atica, Estat\'{\i}stica e
Computa\c{c}\~ao Cient\'{\i}fica, Universidade Estadual de Campinas, Rua Sergio Buarque de Holanda, 651, 13083-859, Campinas - SP, Brazil}
\email{gplanas@unicamp.br}

\thanks{M. Ikeda has been supported by JSPS KAKENHI Grant Number JP 23K03174.  C.J. Niche acknowledges support from  PROEX - CAPES. G. Planas was partially supported by CNPq-Brazil grant 310274/2021-4, and FAPESP-Brazil grant 19/02512-5}



\date{\today}

\begin{abstract}
We examine the energy-critical nonlinear heat equation in the critical space $\dot{H}^1 (\RR ^d)$  for any dimension $d \geq 3$. The aim of this paper is two-fold.  First, we establish a necessary and sufficient condition on initial data at or below the ground state that dichotomizes the behavior of solutions. Specifically, this criterion determines whether the solution will either exist globally with energy decaying to zero over time or blow up in finite time. Secondly, we derive the decay rate for solutions that exist globally.
These results offer a comprehensive characterization of solution behavior for energy-critical conditions in higher-dimensional settings.
\end{abstract}

\maketitle

\section{Introduction}

We consider the Cauchy problem for the energy-critical nonlinear heat equation on $\mathbb R^d$,
\begin{align}
\label{eq:critical-nonlinear-heat}
\begin{cases}
\partial _t u    =  \Delta u +|u|^{2^* - 2} u,&(x,t)\in \mathbb{R}^d\times (0,\infty),\\
u(x,0) = u_0 (x) \in \dot{H}^1 (\mathbb R^d),
\end{cases}
\end{align}
where $d \geq 3$ and  $2^*=2^*(d) =\frac{2d}{d-2}$ is the critical Sobolev exponent.  This equation is energy-critical because the energy
\begin{equation}
\label{eq:energy}
E( u(t)) = \frac{1}{2} \int _{\RR^d}  |\nabla u (t) |^2  \, dx - \frac{1}{2^*} \int _{\RR^d}  |u (t)|^{2^*} \, dx
\end{equation}
is invariant under the natural scaling
\begin{equation}
\label{eq:scaling}
u _{\lambda} (x,t) = \lambda ^{\frac{2}{2^* - 2 }} u (\lambda x, \lambda^2 t) = \lambda ^{\frac{d-2}{2}} u (\lambda x, \lambda^2 t), \quad \lambda > 0.
\end{equation}
Notice that both terms  in \eqref{eq:energy} are invariant under \eqref{eq:scaling}.  This problem is well posed in the sense of Hadamard,  as proved by Chikami, Ikeda and Taniguchi \cite{MR4331259}.  Namely,  for any $u_0 \in \dot{H}^1(\RR^d)$,  there exists a maximal existence time $T_m = T_m (u_0)$ such that \eqref{eq:critical-nonlinear-heat} has a unique mild solution defined in  $I_m = [0,  T_m)$,  which is stable under perturbations of the initial datum.  This solution becomes global in time under conditions which hold,  for example,  if $\|u_0\|_{\dot{H}^1}$ is small enough.  In this case,  we also have
\begin{displaymath}
\lim_{t\rightarrow \infty}\|u(t)\|_{\dot{H}^1}=0.
\end{displaymath}
For a complete statement of this result, see Section \ref{existence-results} and also Theorem 2.1 in Gustafson and Roxanas \cite{MR3765769} for the case $d = 4$.

Many evolution equations with a scale-invariant energy have stationary solutions with important properties.  For \eqref{eq:critical-nonlinear-heat},  such a solution is given by  the unique (up to scaling and translation) positive radial solution to the stationary problem associated with \eqref{eq:critical-nonlinear-heat},  i.e.
\begin{equation}
\label{eqn:stationary problem}
-\Delta W=|W|^{2^*-2}W,
\end{equation}
which is  given by the Aubin-Talenti bubble \cite{MR0431287,MR0463908}
\[
    W(x) = \frac{\left( d(d-2)\right)^{\frac{d-2}{4}}}{(1+|x|^2)^{\frac{d-2}{2}}}.
\]
This ground state $W$ is the extremal of the Sobolev embedding $\dot{H} ^1 (\RR ^d) \subset L^{2^*} (\RR^d)$,  i.e. $W$ realizes the equality in
\begin{displaymath}
\frac{\Vert u \Vert _{L ^{2^*}}}{ \Vert u \Vert _{\dot{H}^1}} \leq C,
\end{displaymath}
for some optimal $C > 0$.  We remark that the Pohozaev identity $\|\nabla W\|_{L^2}^2=\|W\|_{L^{2^*}}^{2^*}$ holds true for $W$.

In several significant examples of energy-critical equations, the ground state acts as a threshold that determines global existence and the dynamic behavior of solutions. This is notably true for the focusing nonlinear Schr\"odinger and nonlinear wave equations, where Kenig and Merle \cite{MR2257393}, \cite{MR2461508} established a crucial dichotomy: if the initial data has energy and $\dot{H}^1$ norm below those of the ground state, the solution is global in time and scatters,  while if the initial data has energy below that of the ground state but a larger $\dot{H}^1$ norm, then the solution experiences finite-time blowup in this norm. Their proof relies on sophisticated concentration-compactness and rigidity arguments, alongside the identification of so-called critical elements that drive this dichotomy through a contradiction argument.

For  nonlinear heat equations of the form
\[ u_t  =  \Delta u +|u|^{p-1} u, \ \quad p > 1 \]
a central question is whether solutions persist globally in time or if their norm becomes unbounded in finite time, resulting in finite-time blowup.  These equations have been extensively studied in the literature under various assumptions on the nonlinearity and initial data.  For a comprehensive overview of the problem, we refer readers to the recent book by Quittner and Souplet \cite{MR3967048}.

The critical exponent case of the nonlinear heat equation is more complicated due to the delicate balance between nonlinearity and dissipation, the increased sensitivity to initial data, and the mathematical challenges of analyzing the behavior of solutions.  Different aspects of this problem,  such as global existence,  finite and infinite time blowup,  as well as other dynamical properties of solutions to \eqref{eq:critical-nonlinear-heat} have been studied by many authors, for example  Collot,  Merle and Raphael \cite{MR3623259}, Cort\'azar, del Pino and Musso \cite{MR4046015},  del Pino,  Musso and Wei \cite{MR3952701}, Fila and King \cite{MR3004680},  Galaktionov and King \cite{MR1968320},  Harada \cite{harada2025dynamicsneargroundstate}, Mizoguchi and Yanagida \cite{MR1464136}, \cite{MR1621030},  Schweyer \cite{MR2990063} and references therein.

We now state the first main result in this article.

\begin{theorem}
\label{thm:global-dynamics}
Let $d\ge 3$ and $u\in C([0,T_m);\dot{H}{^1(\mathbb{R}^d)})$ be the unique local solution to \eqref{eq:critical-nonlinear-heat} with initial data $u_0\in \dot{H}^1(\mathbb{R}^d)$. Assume $E(u_0)\le E(W)$. Then the following statement holds:
\begin{enumerate}
\renewcommand{\theenumi}{\Roman{enumi}}
\item \label{enumi} If $\|\nabla u_0\|_{L^2}<\|\nabla W\|_{L^2}$, then $u$ is dissipative, that is $T_m=\infty$ and
\begin{displaymath}
\lim_{t\rightarrow \infty}\|u(t)\|_{\dot{H}^1}=0.
\end{displaymath}
\item  \label{enumii} If $\|\nabla u_0\|_{L^2}>\|\nabla W\|_{L^2}$ and $u_0\in L^2(\mathbb{R}^d)$, then  the $\dot{H}^1$ norm of $u$ blows up in finite time.
\end{enumerate}
\end{theorem}

In Theorem  \ref{thm:global-dynamics} we have thus established that for the critical exponent in the nonlinearity in \eqref{eq:critical-nonlinear-heat}, the dynamical behaviour of solutions in the critical space $\dot{H}{^1(\mathbb{R}^d)}$ is described by a dichotomy ``dissipativity versus blowup'', for any $d \geq 3$. For $d = 4$,  Gustafson and Roxanas \cite{MR3765769} proved the dissipativity in Theorem \ref{thm:global-dynamics} \eqref{enumi} and besides,  assuming the strict inequality $E(u_0) < E(W)$,   proved Theorem \ref{thm:global-dynamics} \eqref{enumii}. They also indicated that these results should hold for any dimension $d \geq 3$.  In Theorem  \ref{thm:global-dynamics}, we achieve this, and moreover we show that solutions with the same energy as the ground state blow up if they have larger $\dot{H}^1$ norm.  Note that due to the Sobolev Embedding,  conditions $E(u_0)\le E(W)$ and $\|\nabla u_0\|_{L^2}<\|\nabla W\|_{L^2}$ determine a non-empty set. For proving the dissipativity of $u$  we follow ideas along the lines of those in Theorem 1.1 in Chikami, Ikeda and Taniguchi \cite{MR4331259},  where the authors proved an analogous result for the critical Hardy-Sobolev parabolic equation,  i.e.
\begin{displaymath}
\partial _t u    =  \Delta u +|x|^{- \gamma} |u|^{2^* (\gamma) - 2} u,  \quad (x,t) \in \mathbb{R}^d\times (0,\infty),\\
\end{displaymath}
where $0 < \gamma < 2$ and $2^*(\gamma) = \frac{2(d-\gamma)}{d-2}$.  Notice that unlike solutions to \eqref{eq:critical-nonlinear-heat},  solutions to this equation are not translation invariant,  which leads to the need for  a nontrivial adaptation of the ideas in Chikami, Ikeda and Taniguchi \cite{MR4331259}.  The strategy of the proof is based on the principles of concentration-compactness and rigidity,  see Kenig and Merle \cite{MR2257393}, \cite{MR2461508}.  We highlight that the rigidity argument of Chikami, Ikeda and Taniguchi \cite{MR4331259}  is simpler and more flexible than those found in existing literature, as it avoids the need for the backward uniqueness of parabolic equations, cf. Gustafson and Roxanas \cite{MR3765769}.

We now compare Theorem \ref{thm:global-dynamics} to previously known results. In the more general context of $p$-Laplacians, Ishiwata \cite{Ishiwata2008} proved dissipativity in the $\dot{H}^1$ norm and blowup in the $L^{\infty}$ norm for solutions to \eqref{eq:critical-nonlinear-heat}, but under the assumption $u_0 \in H^1 (\mathbb{R}^d) \cap L^{\infty} (\mathbb{R}^d)$. Ikehata, Ishiwata and Suzuki \cite{IkehataIshiwataSuzuki2010} and Ishiwata and Suzuki \cite{IshiwataSuzuki2013} studied dynamical behaviour of solutions in the $L^{\infty}$ norm through dichotomy for initial data in the space $H^1 (e^{|x|^2/4} \, dx), d \geq 3$ (see also Fang and Zhang \cite{MR4400921}). The problem of blow-up and grow-up of solutions in the $L^{\infty}$ norm in various dimensions has been addressed by del Pino,  Musso and Wei \cite{MR4047646},  \cite{MR4307216},  del Pino,  Musso, Wei and Zhou \cite{MR4097503}, Filippas, Herrero and Vel\'azquez \cite{MR1843848}, Harada \cite{MR4072807}, \cite{MR4151635},  Li, Wei, Zhang and Zhou \cite{LI2024113594},  and Wei, Zhang and Zhou \cite{MR4725137}. 

In order to state our second main Theorem,  we need to recall the idea of decay character $r^{\ast} = r^{\ast} (f)$ of $f \in L^2 (\RR^d)$,  which roughly speaking says that  $|\widehat{f} (\xi)| \sim |\xi|^{r^{\ast}}$ near $\xi = 0$. This quantity leads to upper and lower bounds for the decay of solutions to dissipative linear equations with initial datum $f$, and this,  in turn,  is an important tool for proving decay of solutions to dissipative nonlinear equations; for details, see Section   \ref{DecayCharacter}.

\begin{theorem}
\label{thm:decay}
Let $ d\geq 3 $ and $u$ be a dissipative solution as in Theorem \ref{thm:global-dynamics} \eqref{enumi}.  Let $q^{\ast} = r^{\ast} \left( \Lambda u_0\right) > -\frac{d}{2}$,  where $\Lambda = (- \Delta) ^{1/2}$.  Then  we have
\[
\Vert u (t) \Vert ^2 _{\dot{H} ^1} \leq \left\{ \begin{array}{ll}C (1+t) ^{- \min \left\{ \frac{d}{2} + q^{\ast}, 1 \right\}} , & d \leq 10\\
C [\ln (e+t)]^{-2}, \, & d > 10,
 \end{array}\right.
\]
for large enough $t$.
\end{theorem}

For proving Theorem \ref{thm:decay} we follow,  with suitable nontrivial modifications, the ideas  in Kosloff,  Niche and Planas \cite{MR4743818},  who proved the result for $d = 4$.   Roughly speaking,  we show that the $\dot{H}^1$ norm is a Lyapunov function for solutions to  \eqref{eq:critical-nonlinear-heat} and then use an appropriate version of the Fourier Splitting method to prove decay for this critical norm. See also  Ikeda,  Kosloff,  Niche and Planas \cite{MR4767596} for similar methods and results for solutions to the 3D Navier-Stokes and Navier-Stokes-Coriolis equations.  Notice that for $d > 10$ we can only prove an inverse logarithmic decay rate.  This is the first decay estimate we obtain for any $d \geq 3$,  see \eqref{eqn:first-decay}, which we then use to bootstrap in order to obtain faster decay rates.  However,   the nonlinearity
\begin{displaymath}
|u|^{2^* - 2} u = |u|^{\frac{4}{d-2}} u
\end{displaymath}
has such a structure that for $d > 10$ we cannot obtain improvements in the pointwise estimates needed.

\begin{remark}
Notice that Theorem \ref{thm:decay} also holds for the defocusing critical nonlinear heat equation
\begin{align*}
\begin{cases}
\partial _t u    =  \Delta u - |u|^{2^* - 2} u,&(x,t)\in \mathbb{R}^d\times (0,\infty),\\
u(x,0) = u_0 (x) \in \dot{H}^1 (\mathbb R^d),
\end{cases}
\end{align*}
as the proof for decay is essentially identical to that in Section \ref{proof-decay}. We know that any initial data in $\dot{H}^1 (\mathbb R^d)$ leads to a dissipative solution for this equation, see Chikami, Ikeda and Taniguchi \cite[Theorem 4.1]{MR4331259}. 
\end{remark}

This article is organized as follows. In Section \ref{settings}, we gather basic results necessary for our development. More precisely, in Section \ref{existence-results}, we review local and global existence results and energy equalities for solutions to \eqref{eq:critical-nonlinear-heat}. In Section \ref{DecayCharacter}, we introduce key definitions and establish results required to demonstrate the decay rate for dissipative solutions. Section \ref{proof-global-dynamics} is dedicated to proving Theorem \ref{thm:global-dynamics}, and in Section \ref{proof-decay}, we prove Theorem \ref{thm:decay}. Finally, in Appendix \ref{appendix-technical}, we compile technical results essential for proving our main theorems.

 \section{Settings} \label{settings}

\subsection{Existence of solutions and its properties} \label{existence-results}

 We consider the integral form of problem \eqref{eq:critical-nonlinear-heat}
\begin{equation}\label{integral-equation}
	u(t,x) =e^{t\Delta} u_0(x) + \int_{0}^t e^{(t-\tau)\Delta}  |u(\tau,x)|^{2^*-2} u(\tau,x)\,d\tau,
\end{equation}
where
$\{e^{t\Delta}\}_{t>0}$ is the free heat semigroup, defined by
\[
	e^{t\Delta} f(x) = (G(t, \cdot) * f)(x) = \int_{\RR^d} G(t,x-y)f(y)\, dy,\quad t>0,\ x\in \RR^d,
\]
and $G:(0,\infty)\times\mathbb{R}^d\rightarrow (0,\infty)$ is the heat kernel, i.e.,
\[	G(t,x) = (4 \pi t)^{-\frac{d}{2}} e^{-\frac{|x|^2}{4t}},\quad t>0,\  x\in \mathbb R^d.
\]

We introduce an auxiliary space as follows.  Let $T \in (0,\infty]$, $q\in [1,\infty]$, and $\alpha\in \RR$. The space $\mathcal{K}^{q,\alpha}(T)$ is defined by
\[
   \mathcal{K}^{q,\alpha}(T):=\left\{u\in \mathcal{D}'((0,T)\times\mathbb R^d)\ ;\ \|u\|_{\mathcal{K}^{q,\alpha}(T')}
   <\infty\ \text{for any }T' \in (0,T)\right\}
\]
endowed with
\[
\|u\|_{\mathcal K^{q,\alpha}(T)}
	=\sup_{0\le t\le T}t^{\frac{d}{2}(\frac{1}{2^*}-\frac{1}{q})+\alpha}\|u\|_{L^q},
\]
where $ \mathcal{D}'((0,T)\times\mathbb R^d)$ is the space of distributions on $[0,T)\times\mathbb R^d$.
We simply write $\mathcal{K}^{q}(T)=\mathcal{K}^{q,0}(T)$ when $\alpha=0$, and $\mathcal{K}^{q,\alpha}=\mathcal{K}^{q,\alpha}(\infty)$ and $\mathcal{K}^{q}=\mathcal{K}^{q}(\infty)$
when $T=\infty$ if they do not cause confusion.

We recall the existence result for mild solutions to \eqref{eq:critical-nonlinear-heat} obtained by Chikami, Ikeda and Taniguchi \cite{MR4331259}.

\begin{proposition}[Well-posedness in the energy space]
\label{prop:wellposed1}
Assume that $q\in (1,\infty)$ satisfies
\begin{equation}\label{l:crtHS.nonlin.est:c2}
	\frac{1}{2^*} - \frac{1}{d(2^*-1)}
	< \frac{1}{q} < \frac{1}{2^*}.
\end{equation}
Then, the following statements hold:
\begin{itemize}
\item[(i)] (Existence)
For any $u_0 \in \dot{H}^1(\RR^d)$, there exists a maximal existence time $T_m=T_m(u_0)\in (0,\infty]$ such that
there exists a unique mild solution
\[
	u\in C([0,T_m); \dot{H}^1(\RR^d))\cap \mathcal K^q(T_m)
\]
to \eqref{eq:critical-nonlinear-heat} with $u(0)=u_0$. Moreover, the solution $u$ satisfies
\[
\|u\|_{\mathcal K^{\tilde{q}}_{\tilde{r}}(T)} =
	\left(
		\int_0^T (t^{\kappa} \|u(t)\|_{L^{\tilde{q}}})^{\tilde{r}} \, dt
	\right)^\frac{1}{\tilde{r}} < \infty
\]
for any $T \in (0,T_m)$ and for any $\tilde{q}, \tilde{r} \in [1,\infty]$ satisfying \eqref{l:crtHS.nonlin.est:c2} and
\begin{equation}\label{condi-new}
	0\le  \frac1{\tilde r}  < \frac{d}2 \left(\frac1{2^*} - \frac1{\tilde q} \right),
\end{equation}
where $\kappa$ is given by
\[
\kappa = \kappa(\tilde{q},\tilde{r}) = \frac{d}{2} \left( \frac{1}{2^*} - \frac{1}{\tilde q}\right) - \frac{1}{\tilde r}.
\]

\item[(ii)] (Uniqueness in $\mathcal{K}^{q}(T)$)
Let $T>0.$ If $u_1, u_2 \in \mathcal{K}^{q}(T)$
satisfy the integral equation \eqref{integral-equation}
with $u_1(0)=u_2(0)=u_0$, then $u_1=u_2$ on $[0,T].$

\item[(iii)] (Continuous dependence on initial data)
The map $T_m : \dot{H}^1(\RR^d) \to (0,\infty]$ is
lower semicontinuous. Furthermore, for any $u_0, v_0 \in \dot{H}^1(\RR^d)$ and for any $T < \min\{T_m(u_0),T_m(v_0)\}$, there exists a constant $C>0$, depending on $\|u_0\|_{\dot H^1}$, $\|v_0\|_{\dot H^1}$, and $T$, such that
\[
	\sup_{t \in [0,T]} \|u(t) - v(t)\|_{\dot H^1}
	+
	\|u-v\|_{\mathcal K^q(T)} \le C \|u_0 - v_0\|_{\dot H^1}.
\]

\item[(iv)] (Blow-up criterion)
If $T_m < + \infty,$ then $\|u\|_{\mathcal{K}^q(T_m)} = \infty.$

\item[(v)] (Small-data global existence and dissipation)
There exists $\rho >0$ such that
if $u_0 \in \dot{H}^1(\RR^d)$ satisfies
$\|e^{t\Delta} u_0\|_{\mathcal{K}^{q}}\le \rho$,
then $T_m=+\infty$ and
\[
\|u\|_{\mathcal{K}^{q}} \le 2\rho \quad \text{and}\quad \lim_{t\to\infty}\|u(t)\|_{\dot H^1} = 0.
\]

\item[(vi)] (Dissipation of global solutions)
The following statements are equivalent:
\begin{itemize}
	\item[(a)] $T_m=+\infty$ and $\|u\|_{\mathcal K^q} < \infty$.
	\item[(b)] $\lim_{t\to T_m} \|u(t)\|_{\dot H^1}=0$.
	\item[(c)] $\lim_{t\to T_m} t^{\frac{d}{2} (\frac{1}{2^*}-\frac{1}{q})} \|u(t)\|_{L^{q}}=0$.
\end{itemize}

\item[(vii)] Let $d=3$. Suppose that $q$ satisfies the additional assumption
\[	\frac{1}{2^*} - \frac{1}{24}
	< \frac{1}{q}.\]
Then, for any $u_0 \in \dot H^1(\RR^3)$, there exists a maximal existence time $T_m=T_m(u_0)\in (0,\infty]$ such that
there exists a unique mild solution
\[
	u\in C([0,T_m); \dot{H}^1(\RR^d))\cap \mathcal K^q(T_m)\quad \text{and}\quad \partial_t u \in \mathcal K^{3,1}(T_m)
\]
to \eqref{eq:critical-nonlinear-heat} with $u(0)=u_0$. Furthermore, the solution $u$ satisfies
\[
	\partial_t u \in \mathcal K^{2,1}(T_m).
\]
\end{itemize}
\end{proposition}

Note that the definition of $\mathcal{K}^q$,  standard estimates of heat semigroup and the Sobolev Embedding Theorem lead to
\begin{displaymath}
 \| e^{t\Delta} u_0 \|_{\mathcal{K}^q} \leq C \Vert u_0 \Vert _{L ^{2^*}} \leq C \Vert u_0 \Vert _{\dot{H}^1},
\end{displaymath}
hence to the existence of a global solution for small enough $\Vert u_0 \Vert _{\dot{H}^1}$.

The next energy identity  is crucial for proving Theorem \ref{thm:global-dynamics}.

\begin{proposition} \cite[Proposition 2.9]{MR4331259}
\label{prop:energy-id}
Let $u_0 \in \dot{H}^1(\RR^d)$ and
$t_0 \in (0,T_m)$. Then, the mild solution $u$ to \eqref{eq:critical-nonlinear-heat} with
$u(0)=u_0$ satisfies the energy identity
\begin{equation}\label{eqn:energy-id}
	E(u(t))+ \int_{t_0}^t\int_{\RR^d}|\partial_t u(\tau,x)|^2\, dx d\tau
		= E(u(t_0))
\end{equation}
for any $t \in [t_0, T_m)$.
Furthermore, the energy inequality
\begin{equation}\label{energy-ineq}
	E(u(t)) \le E(u_0)
\end{equation}
holds for any $t \in [0, T_m)$.
\end{proposition}

Moreover, as stated in (2.14) in the proof of Proposition 2.9 in Chikami, Ikeda and Taniguchi  \cite{MR4331259},  solutions are such that for any $t_0 \in (0,T_m)$
\begin{displaymath}
\partial_t u,  \Delta u \in L^2 _{loc} \left( [t_0, T_m),  L^2 (\RR^d) \right),
\end{displaymath}
which leads to
\begin{equation}
\label{eqn:nonlinear-term-lo-integrable}
|u|^{\frac{4}{d-2}} u \in L^2 _{loc} \left( [t_0, T_m),  L^2 (\RR^d) \right).
\end{equation}

\subsection{Decay character} \label{DecayCharacter} As stated in the Introduction,  the decay character of an $L^2$ function describes its behaviour near the origin in frequency space,  i.e.   $v_0 \in L^2(\RR^d)$ has decay character $r^{\ast} = r^{\ast} (v_0)$  if $\widehat{v_0} (\xi) \sim |\xi|^{r}$ at $\xi = 0$.  The decay character allows us to obtain upper and lower bounds for the $L^2$ norm of solutions to the heat equation with initial data $v_0$.

We now recall the definition and properties of the decay character,  see Bjorland and M.E. Schonbek \cite{MR2493562},  Niche and M.E. Schonbek \cite{MR3355116},  and Brandolese \cite{MR3493117}.

\begin{definition} \label{decay-indicator}
Let  $v_0 \in L^2(\RR^d)$. For $r \in \left(- \frac{d}{2}, \infty \right)$, we define the {\em decay indicator}  $P_r (v_0)$ corresponding to $v_0$ as
\begin{displaymath}
P_r(v_0) = \lim _{\rho \to 0} \rho ^{-2r-d} \int _{B(\rho)} \bigl |\widehat{v_0} (\xi) \bigr|^2 \, d \xi,
\end{displaymath}
provided this limit exists. In the expression above,  $B(\rho)$ denotes the ball at the origin with radius $\rho$.
\end{definition}

\begin{definition} \label{df-decay-character} The {\em decay character of $ v_0$}, denoted by $r^{\ast} = r^{\ast}( v_0)$ is the unique  $r \in \left( -\frac{d}{2}, \infty \right)$ such that $0 < P_r (v_0) < \infty$, provided that this number exists. We set $r^{\ast} = - \frac{d}{2}$, when $P_r (v_0)  = \infty$ for all $r \in \left( - \frac{d}{2}, \infty \right)$  or $r^{\ast} = \infty$, if $P_r (v_0)  = 0$ for all $r \in \left( -\frac{d}{2}, \infty \right)$.
\end{definition}

In some cases, it is possible to explicitly compute the decay character.  See, for example, M.E. Schonbek \cite{MR837929},  Ferreira, Niche and Planas \cite{MR3565380}.

We now state the Theorem that describes decay in terms of the decay character for linear equations
\begin{equation}
\label{eqn:linear-part}
v_t = \Ll v,  \qquad v(x,0) = v_0 \in L^2 \left( {\RR^d} \right),
\end{equation}
where the differential operator $\Ll$ is diagonalizable,  in the sense that the symbol matrix $\widehat{M}$  in $\widehat{\Ll u} = \widehat{M} \widehat{u}$ is diagonalizable,  with $\widehat{M} = - O^t \cdot |\xi|^{2 \alpha} Id _{\RR^d} \cdot O$,  with $O \in O(d)$.

\begin{theorem} \cite[Theorem 2.10]{MR3355116}
\label{characterization-decay-l2}
Let $v_0 \in L^2 (\RR^d)$ have decay character $r^{\ast} (v_0) = r^{\ast}$. Let $v (t)$ be a solution to  \eqref{eqn:linear-part} with initial datum $v_0$. Then if $- \frac{d}{2 } < r^{\ast}< \infty$, there exist constants $C_1, C_2> 0$ such that
\begin{displaymath}
C_1 (1 + t)^{- \frac{1}{\a} \left( \frac{d}{2} + r^{\ast} \right)} \leq \Vert v(t) \Vert _{L^2} ^2 \leq C_2 (1 + t)^{- \frac{1}{\a} \left( \frac{d}{2} + r^{\ast} \right)}.
\end{displaymath}
\end{theorem}

 In Definitions \ref{decay-indicator} and \ref{df-decay-character},  the limits are assumed to exist,  thus leading to a positive $P_r (u_0)$.  However,  Brandolese \cite{MR3493117} showed  that these may not necessarily exist for all $v_0 \in L^2 (\RR ^d)$ by constructing initial velocities with strong oscillations near the origin in frequency space for which Definition \ref{decay-indicator} does not hold.  He then introduced more general definitions of upper and lower decay character that,  when equal,  recover the  more restrictive one  in Definition \ref{df-decay-character}.    Moreover,  he showed that the decay character exists if and only if $v_0$ belongs to an explicit subset of a certain critical homogeneous Besov space and proved that solutions to the linear system \eqref{eqn:linear-part}  with such  initial data $v_0$ have sharp algebraic decay,  i.e.
if and only if the decay character $r^{\ast} = r^{\ast} (v_0)$ exists.

\section{Proof of Theorem \ref{thm:global-dynamics}}
\label{proof-global-dynamics}

In this section,  we prove Theorem \ref{thm:global-dynamics}, that is,  we establish a dichotomy for the global behavior in time of the solution to \eqref{eq:critical-nonlinear-heat} whose energy is less than or equal to that of the ground state.

The proof of  Theorem \ref{thm:global-dynamics}  \eqref{enumi},  i.e.  dissipativity of solutions,  is carried out in a  similar way to the proof in Chikami, Ikeda and Taniguchi \cite[(a) of Theorem 1.1]{MR4331259},  where the authors  deal with the inhomogeneous energy-critical nonlinearity  $|x|^{-\gamma}|u|^{2^*(\gamma)-1}u$ with $\gamma>0$, where $2^*(\gamma) = \frac{2(d-\gamma)}{d-2}$ is the Hardy-Sobolev critical exponent.  The main modifications involve adapting the arguments from \cite[Lemmas 3.1 and 3.4]{MR4331259}, which rely on the existence of a potential (i.e.  $\gamma>0$).  We omit the proof of Theorem \ref{thm:global-dynamics}  \eqref{enumii},  i.e.  the blowup of solutions whose initial data has larger $\dot{H}^1$ norm than that of the ground state,  because it is an adaptation of the classical argument in Levine \cite{MR348216} (c.f.  Cazenave \cite{cazenave-ENAMA},  Gustafson and Roxanas,  \cite{MR3765769}).

 We introduce now some tools and notation,  which are taken from Chikami, Ikeda and Taniguchi \cite{MR4331259} after setting $\gamma=0$.  The Nehari functional $J:\dot{H}^1(\mathbb{R}^d)\rightarrow \mathbb{R}$ is given by
\[
	J(f) = \frac{d}{d\lambda} E \left( \lambda f \right) \Big|_{\lambda=1}
	= \|f\|_{\dot{H}^1}^2
	- \|f\|_{L^{2^*}}^{2^*}, \ \ \ f\in \dot{H}^1(\mathbb{R}^d).
\]
It is easy to see that if $W$ is a solution to the stationary problem \eqref{eqn:stationary problem}, then $J(W)=0$. Moreover, we have that the ground state $W$ is characterized by the minimization problem
\[
E(W)=\inf \left\{ E(\phi) : \phi\in \dot{H^1}(\mathbb{R}^d)\backslash\{0\},\ J(\phi)=0 \right\}=\frac{1}{d}\|\nabla W\|_{L^2}^2>0.
\]
We note that under the condition $E(\phi)<E(W)$, the inequality $J(\phi)\ge 0$ is equivalent to $\|\nabla \phi\|_{L^2}<\|\nabla W\|_{L^2}$.  First,  assume $\|\nabla \phi\|_{L^2}<\|\nabla W\|_{L^2}$ and suppose that $J(\phi)<0$.  In this case, there exists a $\lambda_0\in (0,1)$ such that $J(\lambda_0\phi)=0$, which implies that $E(W)\le E(\lambda_0 \phi)$. Using the Pohozaev identity we obtain $\|\nabla W\|_{L^2}^2<\|\nabla \phi\|_{L^2}^2$, which is a  contradiction.   Now suppose $J(\phi)\ge 0$.  Then
\begin{align*}
    0  \le J(\phi)
    =\frac{2-2^*}{2}\|\nabla\phi\|_{L^2}^2+ 2^* E(\phi)  < \frac{2-2^*}{2}\|\nabla\phi\|_{L^2}^2 + 2^* E(W) =   \frac{2^*}{d}\left( \|\nabla W \|_{L^2}^2 - \|\nabla\phi\|_{L^2}^2\right ),
\end{align*}
which leads to $\|\nabla \phi\|_{L^2}<\|\nabla W\|_{L^2}$.  Notice that as a result of this inequality  we have that the energy $E$ and the $\dot{H}^1$ norm are equivalent,  as
\begin{equation}
\label{Norm Equivalency}
\Bigl(\frac{1}{2}-\frac{1}{2^*}\Bigr)\|\nabla \phi\|_{L^2}^2\le E(\phi)\le \frac{1}{2}\|\nabla \phi\|_{L^2}^2.
\end{equation}

\begin{proof}[Proof of Theorem \ref{thm:global-dynamics} \eqref{enumi}] It suffices to consider the case $E(u_0)<E(W)$, as the scenario where $E(u_0)=E(W)$ and $J(u_0)\ne 0$ reduces to it.  Indeed, let $u_0 \in \dot H^1(\RR^d)$ with $E(u_0)=E(W)$ and $J(u_0)\ne 0$.  Suppose that there exists a time $t_1\in (0,T_m)$ such that $E(u(t_1)) = E(u_0)$.  Then, by the energy inequality \eqref{energy-ineq}, we have $E(u(t)) = E(u_0)$ for any $t \in [0,t_1]$.  Moreover, for any $t_0 \in (0,t_1)$, the solution $u$ must be stationary in the interval $[t_0,t_1]$  due to the energy identity \eqref{eqn:energy-id}, implying that  $J(u(t)) = 0$ for any $t\in [t_0,t_1]$.  This,  however, contradicts  $J(u_0)\ne 0$ and the continuity of $J(u(t))$ in $t\in [0,t_1]$.  Therefore,  we have $E(u(t)) < E(u_0) = E(W)$ for all $t\in (0,T_m)$.

Let us define a stable set $\mathcal{M}^+$ and an unstable set $\mathcal{M}^-$ in the energy space $\dot{H}^1(\RR^d)$ as
\begin{equation}\nonumber\begin{aligned}
	&\mathcal M^+ = \left\{ \phi \in \dot{H}^1(\mathbb R^d) \ ;\,  E(\phi) < E(W), J(\phi) \ge0 \right\},\\
	&\mathcal M^- = \left\{ \phi \in \dot{H}^1(\mathbb R^d) \ ;\,  E(\phi) < E(W), J(\phi) <0 \right\}.
\end{aligned}\end{equation}
 For $E\in \mathbb{R}$, we introduce a subset $\mathcal{M}_E^+$ of $\dot{H}^1(\mathbb{R}^d)$ defined by
\[
	\mathcal M_E^+ := \left\{\phi \in \dot{H}^1(\mathbb R^d) \,;\,
	E(\phi) < E, J(\phi)\ge0\right\}.
\]
Note that $\mathcal M_E^+$ is not empty when $E>0$, as can be  seen from \eqref{Norm Equivalency}. We also introduce a critical energy $E^c$ given by
\[E^c = \sup\big\{ E \in \RR \,;\, \ \text{$T_m(u_0)=+\infty$ and $\|u\|_{\mathcal K^q}<\infty$} \text{ for any solution $u$ to \eqref{eq:critical-nonlinear-heat} with $u_0\in \mathcal M_E^+$}
\big\},
\]
where $q$ satisfies \eqref{l:crtHS.nonlin.est:c2}.  We note  that when $E < E^c$ all solutions to \eqref{eq:critical-nonlinear-heat} with initial data in $\mathcal M_E^+$ are dissipative, according to (vi) in Proposition \ref{prop:wellposed1}.  Then Theorem \ref{thm:global-dynamics}  \eqref{enumi} is equivalent to $E^c \ge E(W)$.  To establish this, we will proceed  by contradiction.  With this in mind, we assume that
\begin{equation}\label{assum}
	E^c < E(W).
\end{equation}
Our goal is to deduce that  $E^c=0$,  which leads to a contradiction,  as $E^c>0$  is guaranteed by the small-data global existence result in Proposition \ref{prop:wellposed1}.

We begin now the proof of $E^c=0$.  Let  $\{\phi_n\}_{n=1}^\infty \subset \mathcal M^+$ be a sequence  that approaches $E^c$ from above,  this is
\begin{equation}\label{cri-ele}
	\text{$E(\phi_n) \searrow E^c$ as $n\to \infty$}
	\quad \text{and}\quad
	\text{$\|u_n\|_{\mathcal K^q(T_{m}(\phi_n))}=\infty$ for $n=1,2,\cdots$},
\end{equation}
where $u_n$ is a solution to \eqref{eq:critical-nonlinear-heat} with $u_n(0)=\phi_n$.  The following Lemma,  whose proof will be given after this,   provides a  key decomposition of $\{\phi_n\}_{n=1}^\infty$ in terms of a  single profile $\psi$ and  is based on  the linear profile decomposition in  Proposition \ref{prop:profile}.

\begin{lemma}\label{lem:singleprofile}
Suppose (\ref{assum}) holds.  Let $\{\phi_n\}_{n=1}^\infty$ be the above sequence. Then $\phi_n$ can be decomposed as
\begin{equation}\label{decom:J=1}
\phi_n = \psi_n  + w_n,\quad \psi_n (x) =\frac{1}{(\lambda_n)^{\frac{d-2}{2}}} \, \psi\left( \frac{x-x_n}{\lambda_n}\right)
\end{equation}
with translation parameters $\{x_n\}_{n=1}^{\infty}\subset \mathbb{R}^d$ and scale parameters $\{\lambda_n\}_{n=1}^\infty \subset (0,\infty)$, where $\psi_n, w_n \in \mathcal M^+$ and
\begin{equation}\label{w_n}
\lim_{n\to \infty} E(w_n) = \lim_{n\to \infty} \|w_n\|_{\dot{H}^1} = 0.
\end{equation}
\end{lemma}
Let $v^c=v^c(t,x)$ be a solution to \eqref{eq:critical-nonlinear-heat} with $v^c(0) = \psi$, where $\psi$ is given in the decomposition \eqref{decom:J=1}. From \eqref{w_n} and the energy inequality \eqref{energy-ineq}, it follows that
\[
E^c = \lim_{n\to \infty}E(\phi_n) = \lim_{n\to \infty} E(\psi_n)
	= E(\psi) \ge E(v^c(t)),
	\quad t \in [0,T_{m}(\psi)).
\]
Next, we will show that there exists $t_0 \in (0, T_m(\psi))$ such that
\begin{equation}\label{E^c-ineq}
	E^c \le E(v^c(t_0)).
\end{equation}
Suppose  that $E^c > E(v^c(t))$ for any $t \in (0,T_{m}(\psi))$. By the definition of $E^c$ we have that  $T_m(\psi)=+\infty$ and $\|v^c\|_{\mathcal K^q}<\infty$.  Let
\[
v_n^c (t,x) = \frac{1}{(\lambda_n)^{\frac{d-2}{2}}} \, v^c\left(\frac{t}{(\lambda_n)^{2}}, \frac{x-x_n}{\lambda_n}\right).
\]
Each  $v_n^c$ is a solution to \eqref{eq:critical-nonlinear-heat} with $v_n^c(0) = \psi_n$ and satisfies $\|v^c_n\|_{\mathcal K^q}+\|v^c_n\|_{\mathcal K^{\tilde{q}}_{\tilde{r}}}<\infty$ for any $n\in\mathbb N$, where $(\tilde{q},\tilde{r})$ fulfills all the conditions in Proposition \ref{prop:perturbation}.  We apply this  perturbation result with $e=0$  and using \eqref{w_n} we obtain  $T_{m}(\phi_n)=+\infty$ and  $\|u_n\|_{\mathcal K^q}<\infty$
for large enough $n$. This contradicts \eqref{cri-ele}. Hence, there exists a time $t_0 \in (0, T_m(\psi))$
 where  inequality \eqref{E^c-ineq} is satisfied.

In summary, for some $t_0\in (0,T_m(\psi))$
\[
	E^c = E(v^c(t_0)).
\]
By the energy identity \eqref{eqn:energy-id}  we have that $v^c$ is a stationary solution.  However, according to \eqref{assum}, $v^c$
  must be the zero solution because $E^c<E(W)$ and $W$ is the ground state.  This establishes that $E^c = 0$,  which means that  \eqref{assum} cannot hold,  thus proving Theorem \ref{thm:global-dynamics} (\ref{enumi}).
\end{proof}

It remains to prove Lemma \ref{lem:singleprofile}. Before doing so, we state a preliminary result that was established in Chikami, Ikeda and Taniguchi \cite[Lemma 3.4]{MR4331259} for $ \gamma > 0 $. Its proof for $ \gamma = 0 $ follows a similar approach, and we include it here for the sake of completeness.

\begin{lemma}\label{lem:key2}
Let $u^j$ be a solution to \eqref{eq:critical-nonlinear-heat} with $u^j(0)=u^j_0 \in \dot H^1(\RR^d),  d \geq 3$ and $\|u^j\|_{\mathcal K^q} < \infty$ for $j=1,2$.
Introduce
\begin{equation*}\label{sc-tr}
u_n^j(t,x)
:=
\frac{1}{(\lambda_n^j)^{\frac{d-2}{2}}}
u^j\left( \frac{t}{(\lambda_n^j)^2}, \frac{x-x_n^j}{\lambda_n^j}\right)
\end{equation*}
with  scaling parameters $\{\lambda_n^j\}_{n=1}^{\infty} \subset (0,\infty)$ and $\{x_n^j\}_{n=1}^{\infty} \subset \mathbb R^d$. Moreover, we assume that
the scalings are asymptotically orthogonal, in the sense that
\begin{equation}\label{aym-orth}
\frac{\lambda_n^1}{\lambda_n^2} + \frac{\lambda_n^2}{\lambda_n^1} + \frac{|x_n^2 - x_n^1|^2}{\lambda_n^1\lambda_n^2}
\to + \infty \quad \text{ as } n\to \infty.
\end{equation}
Then,
\begin{equation}\label{eq.key}
\lim_{n\to \infty} \sup_{t\in (0,\infty)}t^{\frac{d(2^*-1)}{2}(\frac{1}{2^*}-\frac{1}{q})}
\big\||u_n^1(t)|^{2^*-2}|u_n^2(t)|\big\|_{L^{\frac{q}{2^*-1}}} = 0
\end{equation}
for $q$ as in Proposition \ref{prop:wellposed1}.
\end{lemma}

\begin{proof} We start by noting that 
based on  (iii) in Proposition \ref{prop:wellposed1}, we can assume that $u^j_0 \in \dot H^1(\RR^d) \cap L^\infty(\RR^d)$ for $j=1,2$, without loss of generality.  We introduce the cut-off functions $\chi_k \in C^\infty_0(\mathbb R^d)$  such that $\chi_k(x) \to 1$ as  $k\to\infty$ for each $x\in\mathbb R^d$, and $\eta_k \in C^\infty([0,\infty))$ such that $\eta_k(t) = 1$ for $t\in[0,k]$ and $\eta_k(t)=0$ for $t\in [2k,\infty)$.
In this way, we define   approximating sequences sequences  $\{v_k^j\}_{k=1}^\infty$ and $\{v_{n,k}^j\}_{k=1}^\infty$ of $u^j$ and $u_n^j$, respectively, as
\[
v_k^j(t,x) := \eta_k(t)\chi_k(x)u^j(t,x), \qquad
v_{n,k}^j(t,x) :=\frac{1}{(\lambda_n^j)^{\frac{d-2}{2}}} v_k^j\left( \frac{t}{(\lambda_n^j)^2}, \frac{x-x_n^j}{\lambda_n^j} \right).
\]
With this, to prove \eqref{eq.key} it is enough to show that
\begin{equation}\label{eq,key_k}
\lim_{n\to \infty} \sup_{t\in (0,T)}t^{\frac{d(2^*-1)}{2}(\frac{1}{2^*}-\frac{1}{q})}
\||v_{n,k}^1(t)|^{2^*-2}|v_{n,k}^2(t)|\|_{L^{\frac{q}{2^*-1}}} = 0
\end{equation}
for each $T>0$ and $k\in\mathbb N$. For a more precise argument, see the proof of \cite[Lemmas 3.2 and 3.4]{MR4331259}.

Based on the   asymptotic orthogonality assumption \eqref{aym-orth}, we will examine two cases: either $\lambda_n^1/\lambda_n^2 \to 0$ or $+\infty$, and $|x_n^1 - x_n^2|^2/(\lambda_n^1\lambda_n^2)\to + \infty$.

We first deal with  the case $\lambda_n^1/\lambda_n^2 \to 0$. We perform the change of variables: $s=t/(\lambda_n^1)^2$ and $y=x/\lambda_n^1$ to have
\[
\begin{split}
&
\sup_{t\in (0,T)}t^{\frac{d(2^*-1)}{2}(\frac{1}{2^*}-\frac{1}{q})}
\||v_{n,k}^1(t)|^{2^*-2}|v_{n,k}^2(t)|\|_{L^{\frac{q}{2^*-1}}}\\
&
= \left(\frac{\lambda_n^1}{\lambda_n^2}\right)^{\frac{d-2}{2}}
\sup_{s\in (0, \frac{T}{(\lambda_n^1)^2})}
s^{\frac{d(2^*-1)}{2}(\frac{1}{2^*}-\frac{1}{q})}\Bigg(\int_{\mathbb{R}^d}
\bigg(
\bigg|
v_{k}^1\bigg(s, y - \frac{x_n^1}{\lambda_n^1}
\bigg)\bigg|^{2^*-2}\\
& \qquad\qquad\qquad\qquad\qquad\qquad \times
\bigg|
v_{k}^2
\bigg(
\bigg(
\frac{\lambda_n^1}{\lambda_n^2}
\bigg)^2
s, \frac{\lambda_n^1}{\lambda_n^2}y - \frac{x_n^2}{\lambda_n^2}
\bigg)
\bigg|
\bigg)^{\frac{q}{2^*-1}}\,
dy
\Bigg)^{\frac{2^*-1}{q}},
\end{split}
\]
which goes to zero as $\lambda_n^1/\lambda_n^2 \to 0$ because 
the 
functions $v_{k}^j $  have compact support. 
If $\lambda_n^1/\lambda_n^2 \to +\infty$,
 we introduce the change of variables $ s = t / (\lambda_n^2)^2 $ and $ y = x / \lambda_n^2 $. Under this rescaling, we can proceed by applying the same argument as above in the transformed variables and conclude that \eqref{eq,key_k} holds.

Finally, we address the case where $ |x_n^1 - x_n^2|^2 / (\lambda_n^1 \lambda_n^2) \to +\infty $, which implies that either $ |x_n^1 - x_n^2| / \lambda_n^1 \to +\infty $ or $ |x_n^1 - x_n^2| / \lambda_n^2 \to +\infty $ as $n \to \infty $. 
Without loss of generality, we focus on situation where $ \sup_n (\lambda_n^1 / \lambda_n^2) \in (0, \infty) $ and $ |x_n^1 - x_n^2| / \lambda_n^2 \to +\infty $ as $ n \to \infty $, since the other cases can be treated analogously.

By performing  the change of variable $s=t/(\lambda_n^1)^2$ and  $y=(x-x_n^1)/\lambda_n^1$, there follows
\[
\begin{split}
&
\sup_{t\in (0,T)}t^{\frac{d(2^*-1)}{2}(\frac{1}{2^*}-\frac{1}{q})}
\||v_{n,k}^1(t)|^{2^*-2}|v_{n,k}^2(t)|\|_{L^{\frac{q}{2^*-1}}}\\
&
= \left(\frac{\lambda_n^1}{\lambda_n^2}\right)^{\frac{d-2}{2}}
\sup_{s\in (0, \frac{T}{(\lambda_n^1)^2})}
s^{\frac{d(2^*-1)}{2}(\frac{1}{2^*}-\frac{1}{q})}\Bigg(\int_{\RR^d}
\bigg(
|
v_{k}^1(s, y)|^{2^*-2}\\
& \qquad\qquad\qquad\qquad\qquad\qquad \times
\bigg|
v_{k}^2
\bigg(
\bigg(
\frac{\lambda_n^1}{\lambda_n^2}
\bigg)^2
s, \frac{\lambda_n^1}{\lambda_n^2}y + \frac{x_n^1-x_n^2}{\lambda_n^2}
\bigg)
\bigg|
\bigg)^{\frac{q}{2^*-1}}\,
dy
\Bigg)^{\frac{2^*-1}{q}}.
\end{split}
\]
Since  $v^2_k$ has compact support, the integrand vanishes for sufficiently large $n $. This confirms \eqref{eq,key_k} in all cases, thereby completing the proof of Lemma \ref{lem:key2}.
\end{proof}

We are finally in a position to prove Lemma \ref{lem:singleprofile}.

\begin{proof}[Proof of Lemma \ref{lem:singleprofile}]
The sequence $\{\phi_n\}$ is bounded in $\dot{H}^1(\RR^d)$,  so as a consequence of the  profile decomposition result Proposition \ref{prop:profile},  there exist $J^* \in \{1,2,\ldots,\infty\}$, $\{\psi^j\}_{j=1}^{J^*} \subset \dot{H}^1(\mathbb R^d)$, $\{\lambda_n^j\}_{j=1}^{J^*} \subset (0,\infty)$, and $\{x_n^j\}_{j=1}^{J^*} \subset \mathbb R^d$ such that for $1\le J\le J^*$
\[
	\phi_n(x) = \sum_{j=1}^{J} \, \psi_n^j(x) + w_n^J(x),
\]
and \eqref{profile_1}--\eqref{profile_6} hold, where $\psi_n^j$ is defined by
\begin{equation*}
\psi_n^j(x)
:=
\frac{1}{(\lambda_n^j)^{\frac{d-2}{2}}} \, \psi^j\left( \frac{x-x_n^j}{\lambda_n^j}\right).
\end{equation*}
We will focus on the case where $J^*=\infty$, since the situation where $J^*<\infty$ can be handled in a similar way.
For large $ n$,  from the asymptotic decoupling property \eqref{profile_5} we obtain
\begin{equation*}
	\sum_{j=1}^\infty\| \psi^{j} \|_{\dot{H}^1}
	\le  \|\phi_n\|_{\dot{H}^1} + o_n(1)
	\le C(1 + E^c),
\end{equation*}
from which it follows that $\| \psi^{j} \|_{\dot{H}^1} \to 0$ as $j\to \infty$.  Next,  by using the decoupling properties \eqref{decop-E-phi_n}, \eqref{decop-J-phi_n}, assumption \eqref{assum} and  Lemma \ref{lem:EJ},  we can deduce that
\begin{equation*}
	\psi^j, w_n^J \in \mathcal M^+
	\quad \text{for any $1\le j\le J$ and $n\in\mathbb N$}.
\end{equation*}
We introduce the nonlinear profiles $v^j$  as a solution to \eqref{eq:critical-nonlinear-heat} with initial data $v^j(0)=\psi^j$.  For each $j,n \geq 1$ let
\[
v_n^j (t,x) =
\frac{1}{(\lambda_n^j)^{\frac{d-2}{2}}} v^j\left(\frac{t}{(\lambda_n^j)^{2}}, \frac{x-x_n^j}{\lambda_n^j}\right),
\]
which solves \eqref{eq:critical-nonlinear-heat} with initial data $v^j_n(0)=\psi^j_n$.  Since $\| \psi^{j} \|_{\dot{H}^1} \to 0$ as $j\to \infty$, by the small-data existence result in Proposition \ref{prop:wellposed1},   there exists $J' \in \mathbb N$ such that
\begin{equation}\label{sdge}
\| v^j \|_{\mathcal K^q} \le 2\| \psi^{j} \|_{\dot{H}^1} \quad \text{for any $j \ge J'$.}
\end{equation}

We first  show the existence of at least one bad profile,  i.e., that there exists  $j_0 \in [1,\cdots,J'-1]$ such that
\begin{equation}\label{step2}
\|v^j\|_{\mathcal K^q(T_{m}(\psi^{j_0}))} = \infty.
\end{equation}
To obtain  a contradiction, we assume that for all  $1 \le j \le J'-1$
\begin{equation}\label{contradiction}
\|v^j\|_{\mathcal K^q(T_{m}(\psi^j))} < \infty.
\end{equation}
The well-posedness result in Proposition \ref{prop:wellposed1} implies $T_{m}(\psi^j)= + \infty$ and
\begin{equation*}
\sum_{j=1}^{J'-1}\|v_n^j\|_{\mathcal K^q} =\sum_{j=1}^{J'-1} \|v^j\|_{\mathcal K^q} < \infty.
\end{equation*}
Our aim is to show that  $\|u_n\|_{\mathcal K^q}<\infty$ for a sufficiently large $n$. To this end,  we will apply Lemma \ref{prop:perturbation}. Thereby, we introduce an approximate solution $u_n^J$ by
\begin{equation*}
u_n^J(t) = \sum_{j=1}^{J} v^j_n(t) + e^{t\Delta}w_n^J,
\end{equation*}
So, from \eqref{sdge} and \eqref{profile_1}, we have
\[
\lim_{J\to\infty}\lim_{n\to\infty}\|u_n^J\|_{\mathcal K^q} <\infty.
\]
Moreover, from Proposition \ref{prop:wellposed1} (v) we have
\[
\|v_n^j\|_{\mathcal K^{\tilde{q}}_{\tilde{r}} } = \|v^j\|_{\mathcal K^{\tilde{q}}_{\tilde{r}}} < \infty
\]
and, by using Marcinkiewicz interpolation, we can see that $\| e^{t\Delta}w_n^J\|_{\mathcal K^{\tilde{q}}_{\tilde{r}} }< \infty $, from which we deduce that
\[\|u_n^J\|_{\mathcal K^{\tilde{q}}_{\tilde{r}} } < \infty
\]
for any $n \in \mathbb N$ and for any $(\tilde{q},\tilde{r})$ fulfilling the conditions in Proposition \ref{prop:wellposed1}.  Now, $u_n^J$ is a solution to the approximate equation
\[
\begin{cases}
\partial_t u_n^J - \Delta u_n^J = |u^J_n|^{2^*-2}u^J_n + e_n^J,\\
u_n^J(0) = \phi_n,
\end{cases}
\]
where
\[
e_n^J :=
\sum_{j=1}^J |v^j_n|^{2^*-2}v^j_n - |u^J_n|^{2^*-2}u^J_n .
\]
To apply the perturbation result Lemma \ref{prop:perturbation}, we need to show that
\begin{equation}\label{error-e_n^J}
\lim_{J\to\infty}\limsup_{n\to \infty}\left\|\int_0^t e^{(t-\tau)\Delta}(e_n^J(\tau))\,d\tau\right\|_{\mathcal K^q} = 0.
\end{equation}
We write $e_n^J = e_{n,1}^J  + e_{n,2}^J$,
where $e_{n,1}^J$ and $e_{n,2}^J$ are given by
\[
e_{n,1}^J =  \sum_{j=1}^J |v^j_n|^{2^*-2}v^j_n -  \bigg|\sum_{j=1}^{J} v^j_n\bigg|^{2^*-2}\bigg( \sum_{j=1}^{J}v^j_n\bigg),
\]
\[
e_{n,2}^J =  |u^J_n - e^{t\Delta}w_n^J|^{2^*-2}( u^J_n - e^{t\Delta}w_n^J) - |u^J_n|^{2^*-2}u^J_n.
\]
By using the pointwise estimate
\[
|e_{n,1}^{J}|
\le C
\sum_{\substack{ 1\le i,j \le J \\ i\not = j}}|v_n^i|^{2^*-2}|v_n^j|,
\]
it follows from Lemma \ref{lem:key2} that
\[
\lim_{n\to \infty} \sup_{t\in (0,\infty)}t^{\frac{d(2^*-1)}{2}(\frac{1}{2^*}-\frac{1}{q})}
\||v_n^i(t)|^{2^*-2}|v_n^j(t)|\|_{L^{\frac{q}{2^*-1}}} = 0
\]
for any $1\le i,j \le J$ with $i\ne j$.
As a result, we arrive at
\begin{equation}\label{error-e_n^J_1}
\limsup_{n\to \infty}\left\|\int_0^t e^{(t-\tau)\Delta}(e_{n,1}^J(\tau))\,d\tau\right\|_{\mathcal K^q} = 0,
\end{equation}
for each $J\ge1$.  For the term $e_{n,2}^{J}$, we can bound
\[
\begin{split}
&\left\|\int_0^t e^{(t-\tau)\Delta}(e_{n,2}^J(\tau))\,d\tau\right\|_{\mathcal K^q}\\
& \le C \big(\|u_n^J\|_{\mathcal K^q}^{2^*-2}\|e^{t\Delta}w_n^J\|_{\mathcal K^q}
+ \|e^{t\Delta}w_n^J\|_{\mathcal K^q}^{2^*-2}
\|u_n^J\|_{\mathcal K^q} + \|e^{t\Delta}w_n^J\|_{\mathcal K^q}^{2^*-1}\big),
\end{split}
\]
where the constant $C$ is independent of $J$.  Consequently, by \eqref{profile_1}, we derive
\begin{equation}\label{error-e_n^J_2}
\lim_{J\to\infty}\limsup_{n\to \infty}\left\|\int_0^t e^{(t-\tau)\Delta}(e_{n,2}^J(\tau))\,d\tau\right\|_{\mathcal K^q} = 0.
\end{equation}
Then,  from \eqref{error-e_n^J_1}--\eqref{error-e_n^J_2} we obtain \eqref{error-e_n^J}.
Therefore, we can apply Proposition \ref{prop:perturbation} to $u_n$ and $u_n^J$, leading to  $\|u_n\|_{\mathcal K^q}<\infty$ for a sufficiently large $n$.
This contradicts \eqref{cri-ele}. Consequently, \eqref{contradiction} is negated, and there exists $j_0 \in [1,\cdots,J'-1]$ such that \eqref{step2} is satisfied.

We now show there is only one profile $\psi^{j_0}$ that leads to \eqref{step2}. Without loss of generality we assume that there are two such profiles $\psi^{j_{0,1}}, \psi^{j_{0,2}}$.   Since $\psi^j, w_n^J\in \mathcal M^+$, we have that $E(\psi^j)\ge 0$ and $E(w_n^J)\ge0$.  Moreover, by \eqref{decop-E-phi_n},  we have that for any $J\ge1$
\begin{equation}\label{decop-E-phi_n2}
E^c = \lim_{n\to\infty} E(\phi_n) = \sum_{j=1}^J E(\psi^j) + \lim_{n\to\infty} E(w_n^J).
\end{equation}
Then, $E^c \geq E \left( \psi^{j_{0,1}} \right) + E \left( \psi^{j_{0,2}} \right)$.  By \eqref{step2} and the definition of $E^c$,  we get  $E^c \le E(\psi^{j_{0,i}})$,  for $i = 1,2$.  These two inequalities lead to $E(\psi^{j_{0,i}}) \leq 0$,  which implies $E(\psi^{j_{0,i}}) = 0$.  But this means that $\psi^{j_{0,i}} =0$,  so $v^{j_{0,i}} =0$,  for $i = 1,2$, which contradicts \eqref{step2}.  Hence,  there is a single such profile $\psi^{j_0}$,    for which
\[
E^c = E(\psi^{j_0}).
\]
This and \eqref{decop-E-phi_n2} result in  $E(\psi^j) = 0$ for any $j \not = j_0$ and
\[
\lim_{n\to \infty} E(w_n^J) = 0\quad \text{for any $J\ge1$}.
\]
Also, as $\psi^j , w_n^J\in \mathcal M^+$ imply that
\[
\|\psi^j\|_{\dot{H}^1} \le C E(\psi^j),\quad \|w_n^J\|_{\dot{H}^1} \le C E(w_n^J),
\]
 we conclude that $\psi^j = 0$ for any $j \not = j_0$ and
\[
\displaystyle \lim_{n\to \infty} \|w_n^J\|_{\dot{H}^1} = 0\quad \text{for any $J\ge1$}.
\]
This finishes the proof of Lemma \ref{lem:singleprofile}.
\end{proof}

 \section{Proof of Theorem \ref{thm:decay}}
\label{proof-decay}

 The main technique to prove Theorem \ref{thm:decay} is the Fourier splitting method,  developed by Mar\'{\i}a E. Schonbek \cite{MR571048},  \cite{MR775190},  \cite{MR837929}. This approach involves deriving a differential inequality for the
$\dot{H} ^1$ norm, where the right-hand side represents the average of the solution $u$  within a small, time-dependent shrinking ball centered at the origin in frequency space.

 The following result is key in the proof of the decay rates in  Theorem \ref{thm:decay}.

\begin{proposition}
\label{lyapunov}
The critical norm $\Vert \cdot \Vert _{\dot{H} ^1}$ is a Lyapunov function for dissipative solutions to \eqref{eq:critical-nonlinear-heat}, for large enough $t$.
\end{proposition}

\begin{proof} From \eqref{eqn:nonlinear-term-lo-integrable} and Lemma \ref{lemma-bahouri-chemin-danchin} we deduce that for any $t_1 < t_2$ in a compact interval,  we have the energy equality
\begin{displaymath}
\Vert u(t_2) \Vert ^2 _{\dot{H} ^1}  + 2 \int _{t_1} ^{t_2} \Vert \nabla u(\tau) \Vert ^2  _{\dot{H}^1} \, d\tau = \Vert u (t_1) \Vert ^2 _{\dot{H} ^1} + 2  \int _{t_1} ^{t_2} \langle  u(\tau) , |u(\tau)|^{\frac{4}{d-2}} u(\tau) \rangle _{\dot{H}^1} \, d \tau.
\end{displaymath}
By Parseval's identity, we have
\begin{displaymath}
\langle \Lambda u,\Lambda(|u|^{\frac{4}{d-2}}u)\rangle_{L^2}  =\sum_{j=1}^d\langle \partial_{x_j}u,\partial_{x_j}(|u|^{\frac{4}{d-2}}u)\rangle_{L^2} \leq C\|\nabla u\|_{L^{\frac{2d}{d-2}}} \|\nabla (|u|^{\frac{4}{d-2}}u)\|_{L^\frac{2d}{d+2}} .
\end{displaymath}
By computing
\[\partial_{x_j}(|u|^{\frac{4}{d-2}}u)
=
\frac{4}{d-2}|u|^{\frac{4}{d-2}-1}u\partial_{x_j}u+|u|^{\frac{4}{d-2}}\partial_{x_j}u,
\]
we see that \[ |\partial_{x_j}(|u|^{\frac{4}{d-2}}u)| \leq C |u|^{\frac{4}{d-2}}|\partial_{x_j}u| .\]
By the generalized H\"older inequality and the critical Sobolev embedding
\begin{displaymath}
    \||u|^{\frac{4}{d-2}} \partial_{x_j}u\|_{\frac{2d}{d+2}}  \leq \||u|^{\frac{4}{d-2}}\|_{L^{\frac{d}{2}}}\|\partial_{x_j}u\|_{L^{\frac{2d}{d-2}}} \leq C\|u\|_{L^{\frac{2d}{d-2}}}^{\frac{4}{d-2}}\|\Lambda u\|_{\dot{H}^1}  \leq C\|u\|_{\dot{H}^1}^{\frac{4}{d-2}}\|\Lambda u\|_{\dot{H}^1}.
\end{displaymath}
As a result of this
\begin{equation} \label{eq:estimate-nonlinear}
\langle \Lambda u,\Lambda(|u|^{\frac{4}{d-2}}u)\rangle_{L^2}\leq\|u\|_{\dot{H}^1}^{\frac{4}{d-2}}\|\nabla u\|_{\dot{H}^1}^2,
\end{equation}
and therefore by plugging this estimate in the energy equality
\[
\Vert u(t_2) \Vert ^2 _{\dot{H} ^1}  + 2 \Bigl( 1- C  \, \sup_{t \in[t_1,t_2]} \Vert u(t) \Vert^\frac{4}{d-2} _{\dot{H}^1} \Bigr) \int _{t_1} ^{t_2} \Vert \nabla u(\tau) \Vert ^2  _{\dot{H}^1} \, d\tau \leq \Vert u(t_1) \Vert ^2 _{\dot{H} ^1} .
\]
By Proposition \ref{prop:wellposed1}, we can take $ T >0$ large enough such that $\| u(t) \|_{\dot{H}^1}$ is small for all $t \geq T$.  Hence $ 1 \geq C   \sup_{t > T} \Vert u(t) \Vert^\frac{4}{d-2} _{\dot{H}^1}$. Then for any $ T < t_1 < t_2$ we have
\[
\Vert u(t_2) \Vert ^2 _{\dot{H} ^1} \leq \Vert u(t_1) \Vert ^2 _{\dot{H} ^1} \,,
\]
concluding that the $\dot{H}^1$-norm is a Lyapunov function for large values of $t$.
\end{proof}

\begin{proof}[Proof of Theorem \ref{thm:decay}]
We notice that, from Proposition \ref{lyapunov},   the $\dot{H} ^1$ norm is a nonincreasing function and hence has a derivative a.e. Then
\begin{align*}
\notag \frac{d}{dt} \Vert u(t) \Vert^2 _{\dot{H} ^1} & = 2 \langle \Lambda  u(t), \partial_t \Lambda u(t) \rangle  = 2 \left \langle \Lambda  u (t) ,\Lambda  \left( \Delta u (t)  +  |u|^\frac{4}{d-2} u (t) \right) \right  \rangle \\ \notag &
 = -2 \Vert \nabla u (t) \Vert^2 _{\dot{H} ^1} + 2 \left\langle \Lambda  u (t) , \Lambda \left(  |u|^\frac{4}{d-2} u (t)  \right)  \right\rangle \notag \\ & \leq - 2   \left(1 -  C \Vert  u(t) \Vert^\frac{4}{d-2} _{\dot{H} ^1} \right) \Vert \nabla u(t) \Vert^2 _{\dot{H} ^1}  \leq - \widetilde{C} \Vert \nabla u(t) \Vert^2 _{\dot{H} ^1},
\end{align*}
where we have used \eqref{eq:estimate-nonlinear} and, once again,  Proposition \ref{prop:wellposed1} to have a small enough $\| u (t)\|_{\dot{H}^1}$ for $t$ large enough.

 We start by considering a ball $B(t)$ around the origin in frequency space with  continuous, time-dependent radius $r(t)$ such that
\begin{displaymath}
B(t) = \left\{\xi \in \mathbb{R}^d: |\xi| \leq r(t) = \left( \frac{g'(t)}{\widetilde{C} g(t)} \right) ^{\frac{1}{2}}  \right\},
\end{displaymath}
with $g$ an increasing continuous function such that $g(0) = 1$.   As
\begin{displaymath}
- \widetilde{C} \Vert \nabla u(t) \Vert^2 _{\dot{H} ^1} \leq  - \frac{g'(t)}{g(t)} \int _{B(t) ^c} |\xi| |\widehat{u} (\xi,t) |^2 \, d \xi
\end{displaymath}
we then obtain
\begin{equation}
\label{eqn:key-inequality}
\frac{d}{dt}   \left( g(t)  \Vert u (t) \Vert _{\dot{H} ^1} ^2 \right) \leq g'(t)  \int _{B(t)} ||\xi|  \widehat{u} (\xi, t)| ^2 \, d \xi.
\end{equation}
 We first have
\begin{align}
\int _{B(t)} ||\xi|  \widehat{u} (\xi, t)| ^2 \, d \xi & \leq C \int _{B(t)} \left| e^{- t |\xi| ^2} |\xi| \widehat{u_0} (\xi, t) \right|^2 \, d \xi \notag \\ & + C \int _{B(t)} \left( \int _0 ^t  e^{- (t-s) |\xi| ^2} |\xi| \mathcal{F} \left[|u|^{\frac{4}{d-2}} u \right] (\xi, s) \, ds \right) ^2 \, d \xi. \notag
\end{align}
The first term corresponds to the linear part, i.e.  the heat equation,  hence by Theorem \ref{characterization-decay-l2}, it can be estimated as
\begin{equation}
\label{eq:decay-linear-part}
\int _{B(t)} \left| e^{- t |\xi| ^2} |\xi| \widehat{u_0} (\xi, t) \right|^2 \, d \xi \leq C \Vert e^{t \Delta} \Lambda u_0 \Vert^2 _{L^2} \leq C (1 + t) ^{- \left( \frac{d}{2} + q^{\ast} \right)},
\end{equation}
where $q^{\ast} = r^{\ast} \left( \Lambda u_0 \right)$.  For the non-linear term, we have
\begin{align*}
 \int_{B(t)} & \left(\int_{0}^te^{-(t-s)|\xi|^2}|\xi||\mathcal{F}[|u|^{\frac{4}{d-2}}u]|ds\right)^2d\xi
\le \int_{B(t)}\left(\int_{0}^t|\xi||\mathcal{F}[|u|^{\frac{4}{d-2}}u]|ds\right)^2d\xi \\
&=\int_{B(t)}|\xi|^2\int_{0}^t\int_{0}^t|\mathcal{F}[|u|^{\frac{4}{d-2}}u]||\mathcal{F}[|u|^{\frac{4}{d-2}}u]|dsds'd\xi\\
&\le r(t)^2\int_{0}^t\int_{0}^t\int_{B(t)}|\mathcal{F}[|u|^{\frac{4}{d-2}}u]||\mathcal{F}[|u|^{\frac{4}{d-2}}u]|d\xi dsds'\\
&\le r(t)^2\int_{0}^t\int_{0}^t\left(\int_{B(t)}|\mathcal{F}[|u(s)|^{\frac{4}{d-2}}u(s)]|^2d\xi\right)^{\frac{1}{2}}
\left(\int_{B(t)}|\mathcal{F}[|u(s')|^{\frac{4}{d-2}}u(s')]|^2d\xi\right)^\frac{1}{2}dsds'.
\end{align*}
By using H\"older and Hausdorff–Young inequalities,
\begin{displaymath}
    \int_{B(t)}|\mathcal{F}[|u|^{\frac{4}{d-2}}u]|^2d\xi
     \le \left(\int_{B(t)}d\xi\right)^{\frac{2}{d}}\left(\int_{B(t)}|\mathcal{F}[|u|^{\frac{4}{d-2}}u]|^{\frac{2d}{d-2}}d\xi\right)^{\frac{d-2}{d}}
     \le C r(t)^2\|u\|_{\dot{H}^1}^{\frac{2(d+2)}{d-2}}.
\end{displaymath}
Hence,
\begin{equation} \label{eq:non-linear-part}
\int_{B(t)} \left(\int_{0}^te^{-(t-s)|\xi|^2}|\xi||\mathcal{F}[|u|^{\frac{4}{d-2}}u]|ds\right)^2d\xi \le C  r(t)^4\left(\int_0^t\|u(s)\|_{\dot{H}^1}^{\frac{d+2}{d-2}}ds\right)^2.
\end{equation}
By plugging \eqref{eq:decay-linear-part} and \eqref{eq:non-linear-part} into \eqref{eqn:key-inequality}, there follows
\begin{equation}
\label{eqn:estimate}
\frac{d}{dt}   \left( g(t)  \Vert u (t) \Vert _{\dot{H} ^1} ^2 \right)  \leq C g'(t)  (1 + t) ^{- \left( \frac{d}{2} + q^{\ast} \right)}  + C g'(t) r^4 (t)  \left( \int_0 ^t \Vert u(s)  \Vert _{\dot{H}^1} ^{\frac{d+2}{d-2}} \, ds \right) ^2.
\end{equation}
We first obtain an estimate which is valid for any dimension.  We choose $g(t) = \left[\ln (e+t) \right]^3$ to find
\begin{displaymath}
r(t) = \left(\frac{g'(t)}{\widetilde{C}g(t)}\right)^{\frac{1}{2}} = \left(  \frac{3}{\widetilde{C}(e+t) \ln (e+t)} \right) ^{\frac{1}{2}}.
\end{displaymath}
As $\Vert u(t) \Vert _{\dot{H} ^1} \leq C$ for all $ t >0$, hence from \eqref{eqn:estimate} we obtain
\[
\frac{d}{dt}   \left( \left[\ln (e+t) \right]^3  \Vert u (t) \Vert _{\dot{H} ^1} ^2 \right)  \leq C \frac{\left[\ln (e+t) \right]^2}{e+t}  (1 + t) ^{- \left( \frac{d}{2} + q^{\ast} \right)} + C \frac{1}{(e+t)^3}  t ^2.
\]
Since
\begin{align*}
\int _0 ^t \frac{\left[\ln (e+s) \right]^2}{e+s}  (1 + s) ^{- \left( \frac{d}{2} + q^{\ast} \right)} \, ds & \leq C \int _1 ^{\ln (e+s)} z^2 \, e^{-(\frac{d}{2}+ q^{\ast})z} dz \leq C,
\end{align*}
after integrating in time, we arrive at
\begin{equation}
\label{eqn:first-decay}
\Vert u(t) \Vert _{\dot{H} ^1} ^2  \leq C [\ln (e+t)]^{-2}.
\end{equation}
Notice that this estimate is valid for any $d \geq 3$ and any $q^{\ast} > - \frac{d}{2}$.

We now use \eqref{eqn:first-decay} to obtain better decay estimates for $3 \leq d \leq 10$ through a bootstrap argument.  We first prove our result for $d < 9$,  and then for $d = 10$.

Let $ 3 \leq d < 9$.  Take $g(t) = (1 + t) ^{\alpha}$, with $\alpha > \max \left\{ \frac{d}{2} + q^{\ast}, 1 \right\}$, so $ r(t) = \left (\frac{\alpha}{\widetilde{C}(1+t)} \right)^{\frac{1}{2}}$.  Note that $\alpha > 0$.  Then,  from \eqref{eqn:estimate} and \eqref{eqn:first-decay} we obtain,  after integrating,  that
\begin{displaymath}
 \Vert u (t) \Vert _{\dot{H} ^1} ^2  \leq C (1+t) ^{-\alpha} + C (1 + t) ^{- \left( \frac{d}{2} + q^{\ast} \right)} + C (1+t)^{-1}  \int _0 ^t \frac{ \Vert u (s) \Vert _{\dot{H} ^1} ^2}{[\ln (e+s)]^\frac{8}{d-2}} \, ds,
\end{displaymath}
which leads to

\begin{equation}
\label{eqn:previous-estimate}
 \Vert u (t) \Vert _{\dot{H} ^1} ^2  \leq  C (1 + t) ^{- \left( \frac{d}{2} + q^{\ast} \right)} + C(1+t)^{-1}  \int _0 ^t \frac{ \Vert u (s) \Vert _{\dot{H} ^1} ^2}{ [\ln (e+s)]^\frac{8}{d-2}} \, ds.
\end{equation}
Suppose $q^{\ast} >1-\frac{d}{2}$.  Now consider in \eqref{eqn:previous-estimate}
\begin{align}
x(t) & =  \Vert u (t) \Vert _{\dot{H} ^1} ^2,  \quad a(t) = C (1 + t) ^{- \left( \frac{d}{2} + q^{\ast} \right)} \notag \\ b(t) & = C(1+t)^{-1},  \quad k(t) = \frac{1}{[\ln (e+t)]^\frac{8}{d-2}}.  \notag
\end{align}
Note that,  as $ d < 10,$
\begin{align}
\label{eqn:integral-kb1}
\int_s ^t b(r) \, k(r) \, dr & = C\int _s ^t \frac{1}{(1 + r) [\ln (e+r)]^\frac{8}{d-2}} \, dr  \leq C,  \notag \\ \int_0 ^t a(s) \, k(s) \, ds & =C \int _0 ^t \frac{1}{(1 + s) ^{\frac{d}{2} + q^{\ast}} [\ln (e+s)]^\frac{8}{d-2}} \, ds  \leq C,
\end{align}
hence using Proposition \ref{gronwall-1}  in \eqref{eqn:previous-estimate}  we obtain
\begin{equation*}
\Vert u (t) \Vert _{\dot{H} ^1} ^2  \leq  C (1 + t) ^{- \left( \frac{d}{2} + q^{\ast} \right)} + C (1+t)^{-1} \leq  C (1+t)^{-1}.
\end{equation*}
Now we consider $q^{\ast} \leq 1-\frac{d}{2}$. We rewrite (\ref{eqn:previous-estimate}) as
\begin{displaymath}
(1 + t)  \Vert u (t) \Vert _{\dot{H} ^1} ^2  \leq  C (1 + t) ^{1 - \left( \frac{d}{2} + q^{\ast} \right)} + C  \int _0 ^t \frac{ (1 + s) \Vert u (s) \Vert _{\dot{H} ^1} ^2}{(1 + s)  [\ln (e+s)]^\frac{8}{d-2}} \, ds.
\end{displaymath}
Let

\begin{displaymath}
\psi (t) = (1 + t)  \Vert u (t) \Vert _{\dot{H} ^1} ^2,  \, a(t) = C (1 + t) ^{1 - \left( \frac{d}{2} + q^{\ast} \right)}, \, k(t) = \frac{ C}{(1 + t)  [\ln (e+t)]^\frac{8}{d-2}}.
\end{displaymath}
Notice that $q^{\ast} \leq 1-\frac{d}{2}$ implies that $a(t)$ is non-decreasing,  then Proposition \ref{gronwall-2} and \eqref{eqn:integral-kb1} lead to
\begin{displaymath}
(1 + t)  \Vert u (t) \Vert _{\dot{H} ^1} ^2  \leq  C (1 + t) ^{1 - \left( \frac{d}{2} + q^{\ast} \right)}.
\end{displaymath}
Then  for $d < 10$,
\[ \Vert u (t) \Vert ^2 _{\dot{H} ^1} \leq C (1+t) ^{- \min \left\{ \frac{d}{2} + q^{\ast}, 1 \right\}}.\]

We now deal with the case $ d=10$.  From \eqref{eqn:previous-estimate} we have that
\begin{equation}
\label{eqn:previous-estimate-d10}
 \Vert u (t) \Vert _{\dot{H} ^1} ^2  \leq  C (1 + t) ^{- \left( \frac{d}{2} + q^{\ast} \right)} + C(1+t)^{-1}  \int _0 ^t \frac{ \Vert u (s) \Vert _{\dot{H} ^1} ^2}{ \ln (e+s)} \, ds.
\end{equation}
Suppose $q^{\ast} >1-\frac{d}{2}$.  Now consider in \eqref{eqn:previous-estimate-d10}

\begin{align}
x(t) & =  \Vert u (t) \Vert _{\dot{H} ^1} ^2,  \quad a(t) = C (1 + t) ^{- \left( \frac{d}{2} + q^{\ast} \right)} \notag \\ b(t) & = C(1+t)^{-1},  \quad k(t) = \frac{1}{\ln (e+t)}.  \notag
\end{align}
Note that,
\begin{align*}
\int_s ^t b(r) \, k(r) \, dr & = C\int _s ^t \frac{1}{(1 + r) \ln (e+r)} \, dr  \leq C \ln [\ln(e+t)],   \\ \int_0 ^t a(s) \, k(s) \, ds & =C \int _0 ^t \frac{1}{(1 + s) ^{\frac{d}{2} + q^{\ast}} \ln (e+s)} \, ds  \leq C,\notag
\end{align*}
hence using  Proposition \ref{gronwall-1}  in \eqref{eqn:previous-estimate}  we arrive at
\begin{displaymath}
\Vert u (t) \Vert _{\dot{H} ^1} ^2   \leq  C (1 + t) ^{- \left( \frac{d}{2} + q^{\ast} \right)} + C (1+t)^{-1} \int _0 ^t \frac{[\ln(e+s)]^C}{(1 + s) ^{\frac{d}{2} + q^{\ast}} \ln (e+s)}ds  \leq  C (1+t)^{-1}.
\end{displaymath}
Now we consider $q^{\ast} \leq 1-\frac{d}{2}$. We rewrite \eqref{eqn:previous-estimate} as
\begin{displaymath}
(1 + t)  \Vert u (t) \Vert _{\dot{H} ^1} ^2  \leq  C (1 + t) ^{1 - \left( \frac{d}{2} + q^{\ast} \right)} + C  \int _0 ^t \frac{ (1 + s) \Vert u (s) \Vert _{\dot{H} ^1} ^2}{(1 + s)  \ln (e+s)} \, ds.
\end{displaymath}
Let
\begin{displaymath}
\psi (t) = (1 + t)  \Vert u (t) \Vert _{\dot{H} ^1} ^2,  \  a(t) = C (1 + t) ^{1 - \left( \frac{d}{2} + q^{\ast} \right)}, \ k(t) = \frac{ C}{(1 + t)  \ln (e+t)}.
\end{displaymath}
Notice that $q^{\ast} \leq 1-\frac{d}{2}$ implies that $a(t)$ is non-decreasing,  then Proposition \ref{gronwall-2} give us
\begin{displaymath}
(1 + t)  \Vert u (t) \Vert _{\dot{H} ^1} ^2  \leq  C (1 + t) ^{1 - \left( \frac{d}{2} + q^{\ast} \right)} [\ln(e+t)]^C.
\end{displaymath}
So,
\[ \Vert u (t) \Vert ^2 _{\dot{H} ^1} \leq C (1+t) ^{-  \left(\frac{d}{2} + q^{\ast}\right) } [\ln(e+t)]^C.\]
We will use this now to bootstrap again.
Then,  from \eqref{eqn:estimate},  after integrating,  similarly to \eqref{eqn:previous-estimate} we obtain that

\begin{equation*}
 \Vert u (t) \Vert _{\dot{H} ^1} ^2  \leq  C (1 + t) ^{- \left( \frac{d}{2} + q^{\ast} \right)} + C(1+t)^{-1}  \int _0 ^t \Vert u (s) \Vert _{\dot{H} ^1} ^2(1+s) ^{-\frac{1}{2}  \left(\frac{d}{2} + q^{\ast}\right) } [\ln(e+s)]^\frac{C}{2} \, ds.
\end{equation*}
We re-write
\begin{equation*}
(1+t) \Vert u (t) \Vert _{\dot{H} ^1} ^2  \leq  C (1 + t) ^{1- \left( \frac{d}{2} + q^{\ast} \right)} + C  \int _0 ^t (1+s) \Vert u (s) \Vert _{\dot{H} ^1} ^2 \frac{[\ln(e+s)]^\frac{C}{2}}{(1+s) ^{1+\frac{1}{2}  \left(\frac{d}{2} + q^{\ast}\right) }} \, ds.
\end{equation*}
Similarly, as before,
let
\begin{displaymath}
\psi (t) = (1 + t)  \Vert u (t) \Vert _{\dot{H} ^1} ^2,  \ a(t) = C (1 + t) ^{1 - \left( \frac{d}{2} + q^{\ast} \right)}, \ k(t) = \frac{C[\ln(e+t)]^\frac{C}{2}}{(1+t) ^{1+\frac{1}{2}  \left(\frac{d}{2} + q^{\ast}\right) }}.
\end{displaymath}
As
\[ \int_0^t \frac{[\ln(e+r)]^\frac{C}{2}}{(1+r) ^{1+\frac{1}{2}  \left(\frac{d}{2} + q^{\ast}\right) }} dr\leq C,\]
then by applying Proposition \ref{gronwall-2} there follows
\begin{displaymath}
(1 + t)  \Vert u (t) \Vert _{\dot{H} ^1} ^2  \leq  C (1 + t) ^{1 - \left( \frac{d}{2} + q^{\ast} \right)}.
\end{displaymath}
So,  for $d = 10$, we also have
\begin{displaymath}
\Vert u (t) \Vert ^2 _{\dot{H} ^1} \leq C (1+t) ^{- \min \left\{ \frac{d}{2} + q^{\ast}, 1 \right\}}.
\end{displaymath}
This concludes the proof of Theorem.
\end{proof}

\appendix

\section{Technical results}
\label{appendix-technical}

For the reader's convenience, we collect here some technical tools used in the proof of our main Theorems.  We start with  a perturbation result for the nonlinear heat equation.

\begin{proposition}\cite[Proposition 2.7]{MR4331259}  \label{prop:perturbation}
Let $q,\tilde{q}\in [1,\infty]$ and $\tilde{r} \in [1,\infty)$ satisfy \eqref{l:crtHS.nonlin.est:c2} and \eqref{condi-new}, respectively.
Assume that
\begin{equation}\label{new-condition1}
\frac{1}{2^*}
 - \frac{1}{q} < \frac{1}{\tilde q}
\end{equation}
and
\begin{equation}\label{new-condition2}
\frac{1}{\tilde r} < \min \left\{ \frac{d-2}{2} - \frac{d}{2 \tilde q}, \frac{1}{2^*-2} \right\}.
\end{equation}
Let
$v$ satisfy
\[
	\|v\|_{L^\infty([0,\infty); \dot H^1)} +
	\|v\|_{\mathcal K^{\tilde q}_{\tilde r}} \le M,
\]
and the equation
\[
	\partial_t v - \Delta v = |v|^{2^*-2}v + e
\]
with initial data $v(0) = v_0 \in \dot H^1(\RR^d)$, where $e=e(t,x)$ is a function on $(0,\infty) \times \mathbb R^d$.
Then there exist constants $\delta_0=\delta_0(M)>0$ and $C=C(M)>0$ such that
the following assertion holds: if the error term $e$ and a function $u_0 \in \dot H^1(\mathbb R^d)$ satisfy
\[
	\delta := \| u_0 - v_0\|_{\dot H^1} + \left\|\int_0^t e^{(t-s)\Delta}(e(s))\,ds\right\|_{\mathcal K^q} \le \delta_0,
\]
then there exists a unique solution $u$ to \eqref{eq:critical-nonlinear-heat} on $(0,\infty)\times \mathbb R^d$ with $u(0)=u_0$ satisfying
\[
	\|u-v \|_{L^\infty([0,\infty); \dot H^1)\cap \mathcal K^q}\le C \delta.
\]
\end{proposition}

The following linear profile decomposition is crucial to establishing the existence of a critical element.

\begin{proposition}\cite[Proposition 2.14]{MR4331259}
\label{prop:profile}
Let $\{\phi_n\}_{n=1}^\infty$ be a bounded sequence of functions in $\dot{H}^1(\mathbb R^d)$.
Then, after possibly passing to a subsequence {\rm (}in which case, we rename it $\phi_n${\rm )}, there exist $J^* \in \{1,2,\ldots,\infty\}$, $\{\psi^j\}_{j=1}^{J^*} \subset \dot{H}^1(\mathbb R^d)$, $\{\lambda_n^j\}_{j=1}^{J^*} \subset (0,\infty)$, and $\{x_n^j\}_{j=1}^{J^*} \subset \mathbb R^d$ such that for $1\le J \le J^*$
\begin{equation}\label{profile_0}
	\phi_n(x) = \sum_{j=1}^{J} \frac{1}{(\lambda_n^j)^{\frac{d-2}{2}}} \psi^j\left( \frac{x-x_n^j}{\lambda_n^j}\right) + w_n^J(x),
\end{equation}
where $w_n^J \in \dot{H}^1(\mathbb R^d)$ is such that
\begin{equation}\label{profile_1}
	\limsup_{J\to J^*}
	\lim_{n\to\infty} \|e^{t\Delta}w_n^J\|_{\mathcal K^q}
	= 0,
\end{equation}
\begin{equation}\label{profile_2}
	(\lambda_n^j)^{\frac{d-2}{2}} w_n^J(\lambda_n^j x + x_n^j) \rightharpoonup 0
	\quad \text{in } \dot{H}^1(\mathbb R^d) \text{ as }\ n\to\infty
\end{equation}
for any $1\le j \le J$, and
\begin{equation}\label{profile_3}
	x_n^j \equiv 0 \quad \text{or} \quad |x_n^j|\to\infty \text{ and }\frac{|x_n^j|}{\lambda_n^j}\to\infty \text{ as $n\to\infty$}\quad \text{for }1\le j \le J^*.
\end{equation}
Moreover, the scaling and translation parameters are asymptotically orthogonal in the sense that
\begin{equation}\label{profile_4}
	\frac{\lambda_n^j}{\lambda_n^i} + \frac{\lambda_n^i}{\lambda_n^j} + \frac{|x_n^i - x_n^j|^2}{\lambda_n^j\lambda_n^i}
	\to + \infty
\end{equation}
as $n\to \infty$ for any $i\not = j$. Furthermore, for any $1\le J \le J^*$, we have the following decoupling properties{\rm :}
\begin{equation}\label{profile_5}
	\lim_{n\to\infty}\Big|\|\phi_n\|_{\dot{H}^1}^2 - \sum_{j=1}^J \|\psi^j\|_{\dot{H}^1}^2 - \|w_n^J\|_{\dot{H}^1}^2\Big|=0,
\end{equation}
\begin{equation}\label{profile_6}
	\lim_{n\to\infty}\Big|
	\|\phi_n\|_{L^{2^*}}^{2^*} - \sum_{j=1}^J \|\psi^j\|_{L^{2^*}}^{2^*}
	- \|w_n^J\|_{L^{2^*}}^{2^*}\Big| =0.
\end{equation}
Especially,
\begin{equation}\label{decop-E-phi_n}
	\lim_{n\to\infty}\Big|
	E(\phi_n) - \sum_{j=1}^J E(\psi^j) - E(w_n^J)
	\Big| = 0,
\end{equation}
\begin{equation}\label{decop-J-phi_n}
	\lim_{n\to\infty}\Big|
	J(\phi_n) - \sum_{j=1}^J J(\psi^j) - J(w_n^J)
	\Big| = 0
\end{equation}
for any $1\le J \le J^*$.
\end{proposition}

\begin{lemma}\cite[Lemma 2.13]{MR4331259} \label{lem:EJ}
Let $\{f_j\}_{j=1}^J \subset  \dot{H}^1(\mathbb R^d)$.
Suppose that there exist $\varepsilon>0$ and $0<\delta<E(W)$ with $2 \varepsilon < \delta$ such that
\begin{equation*}
	E\bigg( \sum_{j=1}^J f_j\bigg) < E(W) - \delta,\quad E\bigg( \sum_{j=1}^J f_j\bigg) > \sum_{j=1}^J E( f_j) - \varepsilon,
\end{equation*}
\begin{equation*}
	J\bigg( \sum_{j=1}^J f_j\bigg) \ge -\varepsilon,\quad J\bigg( \sum_{j=1}^J f_j\bigg) \le \sum_{j=1}^J J( f_j) + \varepsilon.
\end{equation*}
Then,
\[
	0 \le E (f_j) < E(W)\quad \text{and}\quad J(f_j) \ge0
\]
for any $1 \le j \le J$.
\end{lemma}

In order to prove  Proposition \ref{lyapunov}, we need the following Lemma.

\begin{lemma}  \cite[Lemma 5.10]{MR2768550} \label{lemma-bahouri-chemin-danchin} Let $v$ be the solution in $C \left(  [0,T];\mathcal{S}' (\RR^d) \right)$ to 
\begin{align*}
\partial_t v - \nu \Delta v & =  f \in L^2(0,T; \dot{H}^{s-1}) , \\
v(0) & = v_0 \in \dot{H}^s (\RR^d).
\end{align*}
Then $v \in L^p(0,T; \dot{H} ^{s + \frac{2}{p}}) \cap C([0,T]; \dot{H}^s)$, for all $p \geq 2$ and
\begin{displaymath}
    \Vert v(t) \Vert ^2 _{\dot{H}^s} + 2 \nu \int _0 ^t \Vert \nabla v( \tau) \Vert ^2 _{\dot{H}^s} \, d \tau = \Vert v_0 \Vert ^2 _{\dot{H}^s} + 2 \int_0^t \langle f(\tau), v(\tau)\rangle_{\dot{H}^s)} d \tau .
\end{displaymath}   
\end{lemma}

 We recall two versions of Gronwall's inequality used in the proof of Theorem \ref{thm:decay}.

\begin{proposition}\cite[Theorem 1, p. 356]{MR1190927} \label{gronwall-1} Let $x,k:J \to \RR$ continuous and $a,b: J \to \RR$ Riemann integrable in $J = [\alpha, \beta]$. Suppose that $b, k \geq 0$ in $J$. Then, if
\begin{displaymath}
x(t) \leq a(t) + b(t) \int _{\alpha} ^t k(s) x(s) \, ds, \quad t \in J
\end{displaymath}
then
\begin{displaymath}
x(t) \leq a(t) + b(t) \int _{\alpha} ^t a(s) k(s) \exp \left( \int_s ^t b(r)k(r)  \, dr \right) \, ds, \quad t \in J.
\end{displaymath}
\end{proposition}

\begin{proposition} \cite[Corollary 1.2, p. 4]{MR1171448} \label{gronwall-2} Let $a,k,\psi,:J \to \RR$ continuous in $J = [\alpha, \beta]$ and $ k \geq 0$.  If $a(t)$ is nondecreasing then
\begin{displaymath}
\psi(t) \leq a(t) +  \int _{\alpha} ^t k(s) \psi(s) \, ds, \quad t \in J
\end{displaymath}
implies
\begin{displaymath}
\psi(t) \leq a(t)  \exp \left( \int_{\alpha} ^t k(s)  \, ds \right) \quad t \in J.
\end{displaymath}
\end{proposition}

\end{document}